\documentclass[12pt]{amsart}
\usepackage{amsmath,amsthm,amsfonts,amssymb,mathrsfs}
\usepackage{color}

\usepackage{tikz}

\usepackage{amssymb,cite}

\usepackage{hyperref}
\date{\today}

\usepackage{hyperref}
\input{xy}
 \xyoption{all}
 \xyoption{arc}

\newtheorem{theorem}{Theorem}

\newtheorem{proposition}{Proposition}
\newtheorem{corollary}{Corollary}

\newtheorem{lemma}{Lemma}
\theoremstyle{definition}

\newtheorem{example}{Example}
\newtheorem{remark}{Remark}

\begin{document}

\title[On the semigroup of monoid endomorphisms of the semigroup $\boldsymbol{B}_{\omega}^{\mathscr{F}}$]{On the semigroup of monoid endomorphisms of the semigroup $\boldsymbol{B}_{\omega}^{\mathscr{F}}$ with a two-element family $\mathscr{F}$ of inductive nonempty subsets of $\omega$}
\author{Oleg Gutik and Inna Pozdniakova}
\address{Ivan Franko National University of Lviv, Universytetska 1, Lviv, 79000, Ukraine}
\email{oleg.gutik@lnu.edu.ua, pozdnyakova.inna@gmail.com}



\begin{abstract}
We study the semigroup of non-injective monoid endomorphisms of the semigroup $\boldsymbol{B}_{\omega}^{\mathscr{F}}$ with a two-elements family $\mathscr{F}$ of inductive nonempty subsets of $\omega$. We describe the structure of elements of the semigroup $\boldsymbol{End}^*_0(\boldsymbol{B}_{\omega}^{\mathscr{F}})$ of non-injective monoid endomorphisms of the semigroup $\boldsymbol{B}_{\omega}^{\mathscr{F}}$. In particular we show that its subsemigroup $\boldsymbol{End}^*(\boldsymbol{B}_{\omega}^{\mathscr{F}})$ of  non-injective non-annihilating  monoid endomorphisms of the semigroup $\boldsymbol{B}_{\omega}^{\mathscr{F}}$ is isomorphic to the direct product of the two-element left-zero semigroup and the multiplicative semigroup of positive integers and describe Green's relations on  $\boldsymbol{End}^*(\boldsymbol{B}_{\omega}^{\mathscr{F}})$.
\end{abstract}

\maketitle



We shall follow the terminology of~\cite{Clifford-Preston-1961, Clifford-Preston-1967, Lawson=1998}.
By $\omega$ we denote the set of all non-negative integers, by $\mathbb{N}$ the set of all positive integers, and by $\mathbb{Z}$ the set of all integers.

\smallskip

Let $\mathscr{P}(\omega)$ be  the family of all subsets of $\omega$. For any $F\in\mathscr{P}(\omega)$ and $n\in\mathbb{Z}$ we put $n+F=\{n+k\colon k\in F\}$ if $F\neq\varnothing$ and $n+\varnothing=\varnothing$. A subfamily $\mathscr{F}\subseteq\mathscr{P}(\omega)$ is called \emph{${\omega}$-closed} if $F_1\cap(-n+F_2)\in\mathscr{F}$ for all $n\in\omega$ and $F_1,F_2\in\mathscr{F}$. For any $a\in\omega$ we denote $[a)=\{x\in\omega\colon x\geqslant a\}$.

\smallskip

A subset $A$ of $\omega$ is said to be \emph{inductive}, if $i\in A$ implies $i+1\in A$. Obviously, $\varnothing$ is an inductive subset of $\omega$.

\begin{remark}[\cite{Gutik-Mykhalenych=2021}]\label{remark-1.1}
\begin{enumerate}
  \item\label{remark-1.1(1)} By Lemma~6 from \cite{Gutik-Mykhalenych=2020} a nonempty subset $F\subseteq \omega$ is inductive in $\omega$ if and only $(-1+F)\cap F=F$.
  \item\label{remark-1.1(2)} Since the set $\omega$ with the usual order is well-ordered, for any nonempty inductive subset $F$ in $\omega$ there exists a nonnegative integer $n_F\in\omega$ such that $[n_F)=F$.
  \item\label{remark-1.1(3)} Statement \eqref{remark-1.1(2)} implies that the intersection of an arbitrary finite family of nonempty inductive subsets in $\omega$ is a nonempty inductive subset of  $\omega$.
\end{enumerate}
\end{remark}

A semigroup $S$ is called {\it inverse} if for any
element $x\in S$ there exists a unique $x^{-1}\in S$ such that
$xx^{-1}x=x$ and $x^{-1}xx^{-1}=x^{-1}$. The element $x^{-1}$ is
called the {\it inverse of} $x\in S$. If $S$ is an inverse
semigroup, then the function $\operatorname{inv}\colon S\to S$
which assigns to every element $x$ of $S$ its inverse element
$x^{-1}$ is called the {\it inversion}.


\smallskip

If $S$ is a semigroup, then we shall denote the subset of all
idempotents in $S$ by $E(S)$. If $S$ is an inverse semigroup, then
$E(S)$ is closed under multiplication and we shall refer to $E(S)$ as a
\emph{band} (or the \emph{band of} $S$). Then the semigroup
operation on $S$ determines the following partial order $\preccurlyeq$
on $E(S)$: $e\preccurlyeq f$ if and only if $ef=fe=e$. This order is
called the {\em natural partial order} on $E(S)$. A \emph{semilattice} is a commutative semigroup of idempotents.

\smallskip

If $S$ is an inverse semigroup then the semigroup operation on $S$ determines the following partial order $\preccurlyeq$
on $S$: $s\preccurlyeq t$ if and only if there exists $e\in E(S)$ such that $s=te$. This order is
called the {\em natural partial order} on $S$ \cite{Wagner-1952}.

\smallskip

If $S$ is a semigroup, then we shall denote the Green relations on $S$ by $\mathscr{R}$, $\mathscr{L}$, $\mathscr{J}$, $\mathscr{D}$ and $\mathscr{H}$ (see \cite[Section~2.1]{Clifford-Preston-1961}):
\begin{align*}
    &\qquad a\mathscr{R}b \mbox{ if and only if } aS^1=bS^1;\\
    &\qquad a\mathscr{L}b \mbox{ if and only if } S^1a=S^1b;\\
    &\qquad a\mathscr{J}b \mbox{ if and only if } S^1aS^1=S^1bS^1;\\
    &\qquad \mathscr{D}=\mathscr{L}\circ\mathscr{R}=
          \mathscr{R}\circ\mathscr{L};\\
    &\qquad \mathscr{H}=\mathscr{L}\cap\mathscr{R}.
\end{align*}
The $\mathscr{L}$-class [$\mathscr{R}$-class, $\mathscr{H}$-class, $\mathscr{D}$-class, $\mathscr{J}$-class] of the semigroup $S$ containing the element $a\in S$ will be denoted by $\boldsymbol{L}_a$ [$\boldsymbol{R}_a$, $\boldsymbol{H}_a$, $\boldsymbol{D}_a$, $\boldsymbol{J}_a$].

\smallskip

The bicyclic monoid ${\mathscr{C}}(p,q)$ is the semigroup with the identity $1$ generated by two elements $p$ and $q$ subjected only to the condition $pq=1$. The semigroup operation on ${\mathscr{C}}(p,q)$ is determined as
follows:
\begin{equation*}
    q^kp^l\cdot q^mp^n=q^{k+m-\min\{l,m\}}p^{l+n-\min\{l,m\}}.
\end{equation*}
It is well known that the bicyclic monoid ${\mathscr{C}}(p,q)$ is a bisimple (and hence simple) combinatorial $E$-unitary inverse semigroup and every non-trivial congruence on ${\mathscr{C}}(p,q)$ is a group congruence \cite{Clifford-Preston-1961}.

\smallskip

On the set $\boldsymbol{B}_{\omega}=\omega\times\omega$ we define the semigroup operation ``$\cdot$'' in the following way
\begin{equation}\label{eq-1.1}
  (i_1,j_1)\cdot(i_2,j_2)=
  \left\{
    \begin{array}{ll}
      (i_1-j_1+i_2,j_2), & \hbox{if~} j_1\leqslant i_2;\\
      (i_1,j_1-i_2+j_2), & \hbox{if~} j_1\geqslant i_2.
    \end{array}
  \right.
\end{equation}
It is well known that the bicyclic monoid $\mathscr{C}(p,q)$ is isomorphic to the semigroup $\boldsymbol{B}_{\omega}$ by the mapping $\mathfrak{h}\colon \mathscr{C}(p,q)\to \boldsymbol{B}_{\omega}$, $q^kp^l\mapsto (k,l)$, $k,l\in\omega$ (see: \cite[Section~1.12]{Clifford-Preston-1961} or \cite[Exercise IV.1.11$(ii)$]{Petrich-1984}). Later we identify the bicyclic monoid $\mathscr{C}(p,q)$ with the semigroup $\boldsymbol{B}_{\omega}$ by the mapping $\mathfrak{h}$.

\smallskip

Next we shall describe the construction which is introduced in \cite{Gutik-Mykhalenych=2020}.

Let $\boldsymbol{B}_{\omega}$ be the bicyclic monoid and $\mathscr{F}$ be an ${\omega}$-closed subfamily of $\mathscr{P}(\omega)$. On the set $\boldsymbol{B}_{\omega}\times\mathscr{F}$ we define the semigroup operation ``$\cdot$'' in the following way
\begin{equation}\label{eq-1.2}
  (i_1,j_1,F_1)\cdot(i_2,j_2,F_2)=
  \left\{
    \begin{array}{ll}
      (i_1-j_1+i_2,j_2,(j_1-i_2+F_1)\cap F_2), & \hbox{if~} j_1\leqslant i_2;\\
      (i_1,j_1-i_2+j_2,F_1\cap (i_2-j_1+F_2)), & \hbox{if~} j_1\geqslant i_2.
    \end{array}
  \right.
\end{equation}
In \cite{Gutik-Mykhalenych=2020} is proved that if the family $\mathscr{F}\subseteq\mathscr{P}(\omega)$ is ${\omega}$-closed then $(\boldsymbol{B}_{\omega}\times\mathscr{F},\cdot)$ is a semigroup. Moreover, if an ${\omega}$-closed family  $\mathscr{F}\subseteq\mathscr{P}(\omega)$ contains the empty set $\varnothing$ then the set
$ 
  \boldsymbol{I}=\{(i,j,\varnothing)\colon i,j\in\omega\}
$ 
is an ideal of the semigroup $(\boldsymbol{B}_{\omega}\times\mathscr{F},\cdot)$. For any ${\omega}$-closed family $\mathscr{F}\subseteq\mathscr{P}(\omega)$ the following semigroup
\begin{equation*}
  \boldsymbol{B}_{\omega}^{\mathscr{F}}=
\left\{
  \begin{array}{ll}
    (\boldsymbol{B}_{\omega}\times\mathscr{F},\cdot)/\boldsymbol{I}, & \hbox{if~} \varnothing\in\mathscr{F};\\
    (\boldsymbol{B}_{\omega}\times\mathscr{F},\cdot), & \hbox{if~} \varnothing\notin\mathscr{F}
  \end{array}
\right.
\end{equation*}
is defined in \cite{Gutik-Mykhalenych=2020}. The semigroup $\boldsymbol{B}_{\omega}^{\mathscr{F}}$ generalizes the bicyclic monoid and the countable semigroup of matrix units. It is proven in \cite{Gutik-Mykhalenych=2020} that $\boldsymbol{B}_{\omega}^{\mathscr{F}}$ is a combinatorial inverse semigroup and Green's relations, the natural partial order on $\boldsymbol{B}_{\omega}^{\mathscr{F}}$ and its set of idempotents are described.
Here, the criteria when the semigroup $\boldsymbol{B}_{\omega}^{\mathscr{F}}$ is simple, $0$-simple, bisimple, $0$-bisimple, or it has the identity, are given.
In particularly in \cite{Gutik-Mykhalenych=2020} it is proved that the semigroup $\boldsymbol{B}_{\omega}^{\mathscr{F}}$ is isomorphic to the semigrpoup of ${\omega}{\times}{\omega}$-matrix units if and only if $\mathscr{F}$ consists of a singleton set and the empty set, and $\boldsymbol{B}_{\omega}^{\mathscr{F}}$ is isomorphic to the bicyclic monoid if and only if $\mathscr{F}$ consists of a non-empty inductive subset of $\omega$.

\smallskip

Group congruences on the semigroup  $\boldsymbol{B}_{\omega}^{\mathscr{F}}$ and its homomorphic retracts  in the case when an ${\omega}$-closed family $\mathscr{F}$ consists of inductive non-empty subsets of $\omega$ are studied in \cite{Gutik-Mykhalenych=2021}. It is proven that a congruence $\mathfrak{C}$ on $\boldsymbol{B}_{\omega}^{\mathscr{F}}$ is a group congruence if and only if its restriction on a subsemigroup of $\boldsymbol{B}_{\omega}^{\mathscr{F}}$, which is isomorphic to the bicyclic semigroup, is not the identity relation. Also in \cite{Gutik-Mykhalenych=2021}, all non-trivial homomorphic retracts and isomorphisms  of the semigroup $\boldsymbol{B}_{\omega}^{\mathscr{F}}$ are described. In \cite{Gutik-Mykhalenych=2022} it is proved that an injective endomorphism $\varepsilon$ of the semigroup $\boldsymbol{B}_{\omega}^{\mathscr{F}}$ is the indentity transformation if and only if  $\varepsilon$ has three distinct fixed points, which is equivalent to existence non-idempotent element $(i,j,[p))\in\boldsymbol{B}_{\omega}^{\mathscr{F}}$ such that  $(i,j,[p))\varepsilon=(i,j,[p))$.

\smallskip

In \cite{Gutik-Lysetska=2021, Lysetska=2020} the algebraic structure of the semigroup $\boldsymbol{B}_{\omega}^{\mathscr{F}}$ is established in the case when ${\omega}$-closed family $\mathscr{F}$ consists of atomic subsets of ${\omega}$.

\smallskip

It is well-known that every automorphism of the bicyclic monoid $\boldsymbol{B}_{\omega}$  is the identity self-map of $\boldsymbol{B}_{\omega}$ \cite{Clifford-Preston-1961}, and hence the group $\mathbf{Aut}(\boldsymbol{B}_{\omega})$ of automorphisms of $\boldsymbol{B}_{\omega}$ is trivial. In \cite{Gutik-Prokhorenkova-Sekh=2021} it is proved that the semigroup $\mathrm{\mathbf{End}}(\boldsymbol{B}_{\omega})$ of all endomorphisms of the bicyclic semigroup $\boldsymbol{B}_{\omega}$ is isomorphic to the semidirect products $(\omega,+)\rtimes_\varphi(\omega,*)$, where $+$ and $*$ are the usual addition and the usual multiplication on $\omega$.

\smallskip

In the paper \cite{Gutik-Pozdniakova=2023-a} we study injective endomorphisms of the semigroup $\boldsymbol{B}_{\omega}^{\mathscr{F}}$ with the two-elements family $\mathscr{F}$ of inductive nonempty subsets of $\omega$. We describe the elements of the semigroup $\boldsymbol{End}^1_*(\boldsymbol{B}_{\omega}^{\mathscr{F}})$ of all injective monoid endomorphisms of the monoid $\boldsymbol{B}_{\omega}^{\mathscr{F}}$.
In particular we show that every element of the semigroup $\boldsymbol{End}^1_*(\boldsymbol{B}_{\omega}^{\mathscr{F}})$ has a form either $\alpha_{k,p}$ or $\beta_{k,p}$, where the endomorphism $\alpha_{k,p}$ is defined by the formulae
\begin{align*}
  (i,j,[0))\alpha_{k,p}&=(ki,kj,[0)), \\
  (i,j,[1))\alpha_{k,p}&=(p+ki,p+kj,[1)),
\end{align*}
for an arbitrary positive integer $k$ and any $p\in\{0,\ldots,k-1\}$, and the endomorphism  $\beta_{k,p}$ is defined by the formulae
\begin{align*}
  (i,j,[0))\beta_{k,p}&=(ki,kj,[0)), \\
  (i,j,[1))\beta_{k,p}&=(p+ki,p+kj,[0)),
\end{align*}
an arbitrary positive integer $k\geqslant 2$ and any $p\in\{1,\ldots,k-1\}$.
In \cite{Gutik-Pozdniakova=2023-a} we describe the product of elements of the semigroup $\boldsymbol{End}^1_*(\boldsymbol{B}_{\omega}^{\mathscr{F}})$:
\begin{align*}
  \alpha_{k_1,p_1}\alpha_{k_2,p_2} &=\alpha_{k_1k_2,p_2+k_2p_1}; \\
  \alpha_{k_1,p_1}\beta_{k_2,p_2}  &=\beta_{k_1k_2,p_2+k_2p_1}; \\
  \beta_{k_1,p_1}\beta_{k_2,p_2}   &=\beta_{k_1k_2,k_2p_1}; \\
  \beta_{k_1,p_1}\alpha_{k_2,p_2}  &=\beta_{k_1k_2,k_2p_1}.
\end{align*}
Also, here we prove that Green's relations $\mathscr{R}$, $\mathscr{L}$, $\mathscr{H}$, $\mathscr{D}$, and $\mathscr{J}$  on $\boldsymbol{End}^1_*(\boldsymbol{B}_{\omega}^{\mathscr{F}})$ coincide with the equality relation.

\smallskip

Later we assume that an ${\omega}$-closed family $\mathscr{F}$ consists of two nonempty inductive nonempty subsets of $\omega$.

\smallskip

This paper is a continuation of \cite{Gutik-Pozdniakova=2023-a}.
We study non-injective monoid endomorphisms of the semigroup $\boldsymbol{B}_{\omega}^{\mathscr{F}}$. We describe the structure of elements of the semigroup $\boldsymbol{End}_0^*(\boldsymbol{B}_{\omega}^{\mathscr{F}})$ of all non-injective monoid endomorphisms of the semigroup $\boldsymbol{B}_{\omega}^{\mathscr{F}}$. In particular we show that its subsemigroup $\boldsymbol{End}^*(\boldsymbol{B}_{\omega}^{\mathscr{F}})$ of all non-injective non-annihilating monoid endomorphisms of the semigroup $\boldsymbol{B}_{\omega}^{\mathscr{F}}$ is isomorphic to the direct product the two-element left-zero semigroup and the multiplicative semigroup of positive integers and describe Green's relations on  $\boldsymbol{End}^*(\boldsymbol{B}_{\omega}^{\mathscr{F}})$.

\smallskip


\begin{remark}\label{remark-2.1}
By Proposition~1 of \cite{Gutik-Mykhalenych=2021} for any $\omega$-closed family $\mathscr{F}$ of inductive subsets in $\mathscr{P}(\omega)$ there exists an $\omega$-closed family $\mathscr{F}^*$ of inductive subsets in $\mathscr{P}(\omega)$ such that $[0)\in \mathscr{F}^*$ and the semigroups $\boldsymbol{B}_{\omega}^{\mathscr{F}}$ and $\boldsymbol{B}_{\omega}^{\mathscr{F}^*}$ are isomorphic. Hence without loss of generality we may assume that the family $\mathscr{F}$ contains the set $[0)$.
\end{remark}

\smallskip

If $\mathscr{F}$ is an arbitrary $\omega$-closed family $\mathscr{F}$ of inductive subsets in $\mathscr{P}(\omega)$ and $[s)\in \mathscr{F}$ for some $s\in \omega$ then
\begin{equation*}
  \boldsymbol{B}_{\omega}^{\{[s)\}}=\{(i,j,[s))\colon i,j\in\omega\}
\end{equation*}
is a subsemigroup of $\boldsymbol{B}_{\omega}^{\mathscr{F}}$ \cite{Gutik-Mykhalenych=2021} and by Proposition~3 of \cite{Gutik-Mykhalenych=2020} the semigroup $\boldsymbol{B}_{\omega}^{\{[s)\}}$ is isomorphic to the bicyclic semigroup.

\begin{lemma}\label{lemma-2.2}
Let $\mathscr{F}=\{[0), [1)\}$ and let $\mathfrak{e}$ be a monoid endomorphism of the semigroup $\boldsymbol{B}_{\omega}^{\mathscr{F}}$. If $(i_1,j_1,F)\mathfrak{e}=(i_2,j_2,F)\mathfrak{e}$ for distinct two elements $(i_1,j_1,F),(i_2,j_2,F)$ of $\boldsymbol{B}_{\omega}^{\mathscr{F}}$ for some $F\in\mathscr{F}$ then $\mathfrak{e}$ is the annihilating endomorphism of $\boldsymbol{B}_{\omega}^{\mathscr{F}}$.
\end{lemma}

\begin{proof}
By  Theorem~1 of \cite{Gutik-Mykhalenych=2021} the image $(\boldsymbol{B}_{\omega}^{\mathscr{F}})\mathfrak{e}$ is a subgroup of $\boldsymbol{B}_{\omega}^{\mathscr{F}}$.
By Theorem~4$(iii)$ of \cite{Gutik-Mykhalenych=2020} every $\mathscr{H}$-class in $\boldsymbol{B}_{\omega}^{\mathscr{F}}$ is a singleton, and hence $\mathfrak{e}$ is the annihilating monoid endomorphism of $\boldsymbol{B}_{\omega}^{\mathscr{F}}$.
\end{proof}

\begin{lemma}\label{lemma-2.3}
Let $\mathscr{F}=\{[0), [1)\}$. Then $(\boldsymbol{B}_{\omega}^{\mathscr{F}})\mathfrak{e}\subseteq \boldsymbol{B}_{\omega}^{\{[0)\}}$ for any non-injective monoid endomorphism $\mathfrak{e}$ of $\boldsymbol{B}_{\omega}^{\mathscr{F}}$.
\end{lemma}

\begin{proof}
By Proposition~3 of \cite{Gutik-Mykhalenych=2020} the subsemigroup $\boldsymbol{B}_{\omega}^{\{[0)\}}$ of $\boldsymbol{B}_{\omega}^{\mathscr{F}}$ is isomorphic to the bicyclic semigroup and hence by Corollary 1.32 of \cite{Clifford-Preston-1961} the image $(\boldsymbol{B}_{\omega}^{\{[0)\}})\mathfrak{e}$ either is isomorphic to the bicyclic semigroup or is a cyclic subgroup of $\boldsymbol{B}_{\omega}^{\mathscr{F}}$. Since $(0,0,[0))\mathfrak{e}=(0,0,[0))$, Proposition 4 from \cite{Gutik-Mykhalenych=2021} implies that $(\boldsymbol{B}_{\omega}^{\{[0)\}})\mathfrak{e}\subseteq \boldsymbol{B}_{\omega}^{\{[0)\}}$ in the case when the image $(\boldsymbol{B}_{\omega}^{\{[0)\}})\mathfrak{e}$ is isomorphic to the bicyclic semigroup. In the other case  we have that the equality $(0,0,[0))\mathfrak{e}=(0,0,[0))$ implies that
\begin{equation*}
(\boldsymbol{B}_{\omega}^{\{[0)\}})\mathfrak{e}\subseteq \{(0,0,[0))\}\subseteq \boldsymbol{B}_{\omega}^{\{[0)\}},
\end{equation*}
because by Theorem~4$(iii)$ of \cite{Gutik-Mykhalenych=2020} every $\mathscr{H}$-class in $\boldsymbol{B}_{\omega}^{\mathscr{F}}$ is a singleton.

\smallskip

Next, by Proposition~3 of \cite{Gutik-Mykhalenych=2020} the subsemigroup $\boldsymbol{B}_{\omega}^{\{[1)\}}$ of $\boldsymbol{B}_{\omega}^{\mathscr{F}}$ is isomorphic to the bicyclic semigroup and hence by Corollary 1.32 of \cite{Clifford-Preston-1961} the image $(\boldsymbol{B}_{\omega}^{\{[1)\}})\mathfrak{e}$ either is isomorphic to the bicyclic semigroup or is a cyclic subgroup of $\boldsymbol{B}_{\omega}^{\mathscr{F}}$. Suppose that the image $(\boldsymbol{B}_{\omega}^{\{[1)\}})\mathfrak{e}$ is isomorphic to the bicyclic semigroup and $(\boldsymbol{B}_{\omega}^{\{[1)\}})\mathfrak{e}\subseteq\boldsymbol{B}_{\omega}^{\{[1)\}}$. Then monoid endomorphism $\mathfrak{e}$ of $\boldsymbol{B}_{\omega}^{\mathscr{F}}$ is injective. Indeed, injectivity of the restriction $\mathfrak{e}{\upharpoonleft}_{\boldsymbol{B}_{\omega}^{\{[1)\}}}\boldsymbol{B}_{\omega}^{\{[1)\}}\to \boldsymbol{B}_{\omega}^{\{[1)\}}$, Proposition 4 of \cite{Gutik-Mykhalenych=2021}, Corollary 1.32 of \cite{Clifford-Preston-1961}, Theorem~4$(iii)$ of \cite{Gutik-Mykhalenych=2020}, and the equality $(0,0,[0))\mathfrak{e}=(0,0,[0))$ imply that either the restriction $\mathfrak{e}{\upharpoonleft}_{\boldsymbol{B}_{\omega}^{\{[0)\}}}\boldsymbol{B}_{\omega}^{\{[0)\}}\to \boldsymbol{B}_{\omega}^{\{[0)\}}$ is an injective mapping or is an annihilating endomorphism. In the case when the restriction $\mathfrak{e}{\upharpoonleft}_{\boldsymbol{B}_{\omega}^{\{[0)\}}}\boldsymbol{B}_{\omega}^{\{[0)\}}\to \boldsymbol{B}_{\omega}^{\{[0)\}}$ is an injective mapping we get that the endomorphism $\mathfrak{e}$ is injective. If the image $(\boldsymbol{B}_{\omega}^{\{[0)\}})\mathfrak{e}$ is a singleton then by Lemma~\ref{lemma-2.2} we have that $\mathfrak{e}$ is the annihilating monoid  endomorphism of $\boldsymbol{B}_{\omega}^{\mathscr{F}}$. In the both cases we obtain that $(\boldsymbol{B}_{\omega}^{\mathscr{F}})\mathfrak{e}\subseteq \boldsymbol{B}_{\omega}^{\{[0)\}}$.
\end{proof}

\begin{example}\label{example-2.4}
Let $\mathscr{F}=\{[0), [1)\}$ and $k$ be an arbitrary non-negative integer. We define a map $\gamma_k\colon \boldsymbol{B}_{\omega}^{\mathscr{F}}\to \boldsymbol{B}_{\omega}^{\mathscr{F}}$ by the formulae
\begin{equation*}
  (i,j,[0))\gamma_k=(i,j,[1))\gamma_k=(ki,kj,[0))
\end{equation*}
for all $i,j\in\omega$.

\smallskip

We claim that $\gamma_k\colon \boldsymbol{B}_{\omega}^{\mathscr{F}}\to \boldsymbol{B}_{\omega}^{\mathscr{F}}$ is an endomorphism. Example~2 and Proposition~5 from \cite{Gutik-Mykhalenych=2021} imply that the map $\gamma_1\colon \boldsymbol{B}_{\omega}^{\mathscr{F}}\to \boldsymbol{B}_{\omega}^{\mathscr{F}}$ is a homomorphic retraction of the monoid $\boldsymbol{B}_{\omega}^{\mathscr{F}}$, and hence it is a monoid endomorphism of $\boldsymbol{B}_{\omega}^{\mathscr{F}}$. By Lemma~2 of \cite{Gutik-Prokhorenkova-Sekh=2021} every monoid endomorphism $\mathfrak{h}$ of the semigroup $\boldsymbol{B}_{\omega}$ has the following form
\begin{equation*}
  (i,j)\mathfrak{h}=(ki,kj), \qquad \hbox{for some} \quad k\in\omega.
\end{equation*}
This implies that the map $\gamma_k$ is a monoid endomorphism of $\boldsymbol{B}_{\omega}^{\mathscr{F}}$.
\end{example}

\begin{example}\label{example-2.5}
Let $\mathscr{F}=\{[0), [1)\}$ and $k$ be an arbitrary non-negative integer. We define a map $\delta_k\colon \boldsymbol{B}_{\omega}^{\mathscr{F}}\to \boldsymbol{B}_{\omega}^{\mathscr{F}}$ by the formulae
\begin{equation*}
  (i,j,[0))\delta_k=(ki,kj,[0)) \qquad \hbox{and} \qquad (i,j,[1))\delta_k=(k(i+1),k(j+1),[0))
\end{equation*}
for all $i,j\in\omega$.
\end{example}

\begin{proposition}\label{proposition-2.6}
Let $\mathscr{F}=\{[0), [1)\}$. Then for any $k\in\omega$ the map $\delta_k$ is an endomorphism of the monoid $\boldsymbol{B}_{\omega}^{\mathscr{F}}$.
\end{proposition}

\begin{proof}
Since by Proposition~3 of \cite{Gutik-Mykhalenych=2020} the subsemigroups $\boldsymbol{B}_{\omega}^{\{[0)\}}$ and $\boldsymbol{B}_{\omega}^{\{[1)\}}$ of $\boldsymbol{B}_{\omega}^{\mathscr{F}}$ are isomorphic to the bicyclic semigroup, by Lemma~2 of \cite{Gutik-Prokhorenkova-Sekh=2021} the restrictions $\delta_k{\upharpoonleft}_{\boldsymbol{B}_{\omega}^{\{[0)\}}}\colon \boldsymbol{B}_{\omega}^{\{[0)\}}$ $\to \boldsymbol{B}_{\omega}^{\mathscr{F}}$ and $\delta_k{\upharpoonleft}_{\boldsymbol{B}_{\omega}^{[1)}}\colon \boldsymbol{B}_{\omega}^{\{[1)\}}\to \boldsymbol{B}_{\omega}^{\mathscr{F}}$ of $\delta_k$ are homomorphisms. Hence it sufficient to show that the following equalities
\begin{align*}
  (i_1,j_1,[0))\delta_k\cdot (i_2,j_2,[1))\delta_k&=((i_1,j_1,[0))\cdot (i_2,j_2,[1)))\delta_k; \\
  (i_1,j_1,[1))\delta_k\cdot (i_2,j_2,[0))\delta_k&=((i_1,j_1,[1))\cdot (i_2,j_2,[0)))\delta_k,
\end{align*}
hold for any $i_1,j_1,i_2,j_2\in\omega$.

\smallskip

We observe that the above equalities are trivial in the case when $k=0$. Hence later we assume that $k$ is a positive integer.

\smallskip

Then we have that
\begin{align*}
  (i_1,&j_1,[0))\delta_k\cdot (i_2,j_2,[1))\delta_k=(ki_1,kj_1,[0))\cdot(k(i_2+1),k(j_2+1),[0))= \\
   &=
    \left\{
      \begin{array}{ll}
        (ki_1{+}k(i_2{+}1){-}kj_1,k(j_2{+}1),(kj_1{-}k(i_2{+}1){+}[0))\cap[0)), & \hbox{if~} kj_1<k(i_2+1);\\
        (ki_1,k(j_2+1),[0)\cap[0)),                               & \hbox{if~} kj_1=k(i_2+1);\\
        (ki_1,kj_1{+}k(j_2{+}1){-}k(i_2{+}1),[0)\cap(k(i_2{+}1){-}kj_1{+}[0))), & \hbox{if~} kj_1>k(i_2+1)
      \end{array}
    \right.
    \\
    &=
    \left\{
      \begin{array}{ll}
        (k(i_1+i_2+1-j_1),k(j_2+1),[0)), & \hbox{if~} j_1<i_2+1;\\
        (ki_1,k(j_2+1),[0)),             & \hbox{if~} j_1=i_2+1;\\
        (ki_1,k(j_1+j_2-i_2),[0)),       & \hbox{if~} j_1>i_2+1
      \end{array}
    \right.
    \\
    &=
    \left\{
      \begin{array}{ll}
        (k(i_1+i_2+1-j_1),k(j_2+1),[0)), & \hbox{if~} j_1<i_2;\\
        (k(i_1+i_2+1-j_1),k(j_2+1),[0)), & \hbox{if~} j_1=i_2;\\
        (ki_1,k(j_2+1),[0)),             & \hbox{if~} j_1=i_2+1;\\
        (ki_1,k(j_1+j_2-i_2),[0)),       & \hbox{if~} j_1>i_2+1
      \end{array}
    \right.
    \\
    &=
    \left\{
      \begin{array}{ll}
        (k(i_1+i_2+1-j_1),k(j_2+1),[0)), & \hbox{if~} j_1<i_2;\\
        (k(i_1+1),k(j_2+1),[0)),         & \hbox{if~} j_1=i_2;\\
        (ki_1,k(j_2+1),[0)),             & \hbox{if~} j_1=i_2+1;\\
        (ki_1,k(j_1+j_2-i_2),[0)),       & \hbox{if~} j_1>i_2+1,
      \end{array}
    \right.
\end{align*}
\begin{align*}
  ((i_1,j_1,[0))\cdot (i_2,j_2,[1)))\delta_k&=
   \left\{
     \begin{array}{ll}
       (i_1+i_2-j_1,j_2,(j_1-i_2+[0))\cap[1))\delta_k, & \hbox{if~} j_1<i_2; \\
       (i_1,j_2,[0)\cap[1))\delta_k,                   & \hbox{if~} j_1=i_2; \\
       (i_1,j_1+j_2-i_2,[0)\cap(i_2-j_1+[1)))\delta_k, & \hbox{if~} j_1>i_2
     \end{array}
   \right.
   \\
   &=
   \left\{
     \begin{array}{ll}
       (i_1+i_2-j_1,j_2,[1))\delta_k, & \hbox{if~} j_1<i_2; \\
       (i_1,j_2,[1))\delta_k,         & \hbox{if~} j_1=i_2; \\
       (i_1,j_1+j_2-i_2,[0))\delta_k, & \hbox{if~} j_1>i_2
     \end{array}
   \right.
   \\
   &=
   \left\{
     \begin{array}{ll}
       (k(i_1+i_2-j_1+1),k(j_2+1),[0)), & \hbox{if~} j_1<i_2; \\
       (k(i_1+1),k(j_2+1),[0)),         & \hbox{if~} j_1=i_2; \\
       (ki_1,k(j_1+j_2-i_2),[0)),       & \hbox{if~} j_1>i_2
     \end{array}
   \right.
   \\
   &=
   \left\{
     \begin{array}{ll}
       (k(i_1+i_2-j_1+1),k(j_2+1),[0)), & \hbox{if~} j_1<i_2; \\
       (k(i_1+1),k(j_2+1),[0)),         & \hbox{if~} j_1=i_2; \\
       (ki_1,k(j_1+j_2-i_2),[0)),       & \hbox{if~} j_1=i_2+1; \\
       (ki_1,k(j_1+j_2-i_2),[0)),       & \hbox{if~} j_1>i_2+1
     \end{array}
   \right.
   \\
    &=
    \left\{
      \begin{array}{ll}
        (k(i_1+i_2+1-j_1),k(j_2+1),[0)), & \hbox{if~} j_1<i_2;\\
        (k(i_1+1),k(j_2+1),[0)),         & \hbox{if~} j_1=i_2;\\
        (ki_1,k(j_2+1),[0)),             & \hbox{if~} j_1=i_2+1;\\
        (ki_1,k(j_1+j_2-i_2),[0)),       & \hbox{if~} j_1>i_2+1,
      \end{array}
    \right.
\end{align*}
and
\begin{align*}
  (i_1,j_1,&[1))\delta_k\cdot (i_2,j_2,[0))\delta_k=(k(i_1+1),k(j_1+1),[0))\cdot(ki_2,kj_2,[0))= \\
   &=
   \left\{
     \begin{array}{ll}
       (k(i_1{+}1){+}ki_2{-}k(j_1{+}1),kj_2,(k(j_1{+}1){-}ki_2{+}[0))\cap[0)), & \hbox{if~} k(j_1+1)<ki_2;\\
       (k(i_1+1),kj_2,[0)\cap[0)),                               & \hbox{if~} k(j_1+1)=ki_2;\\
       (k(i_1{+}1),k(j_1{+}1){+}kj_2{-}ki_2,[0)\cap(ki_2{-}k(j_1{+}1){+}[0))), & \hbox{if~} k(j_1+1)>ki_2
     \end{array}
   \right.
   \\
   &=
   \left\{
     \begin{array}{ll}
       (k(i_1+i_2-j_1),kj_2,[0)),       & \hbox{if~} j_1+1<i_2;\\
       (k(i_1+1),kj_2,[0)),             & \hbox{if~} j_1+1=i_2;\\
       (k(i_1+1),k(j_1+1+j_2-i_2),[0)), & \hbox{if~} j_1+1>i_2
     \end{array}
   \right.
   \\
   &=
   \left\{
     \begin{array}{ll}
       (k(i_1+i_2-j_1),kj_2,[0)),       & \hbox{if~} j_1+1<i_2;\\
       (k(i_1+1),kj_2,[0)),             & \hbox{if~} j_1+1=i_2;\\
       (k(i_1+1),k(j_1+1+j_2-i_2),[0)), & \hbox{if~} j_1=i_2;\\
       (k(i_1+1),k(j_1+1+j_2-i_2),[0)), & \hbox{if~} j_1+1>i_2
     \end{array}
   \right.
   \\
   &=
   \left\{
     \begin{array}{ll}
       (k(i_1+i_2-j_1),kj_2,[0)),       & \hbox{if~} j_1+1<i_2; \\
       (k(i_1+1),kj_2,[0)),             & \hbox{if~} j_1+1=i_2; \\
       (k(i_1+1),k(j_2+1),[0)),         & \hbox{if~} j_1=i_2; \\
       (k(i_1+1),k(j_1+j_2-i_2+1),[0)), & \hbox{if~} j_1>i_2,
     \end{array}
   \right.
\end{align*}
\begin{align*}
  ((i_1,j_1,[1))\cdot (i_2,j_2,[0)))\delta_k&=
   \left\{
     \begin{array}{ll}
       (i_1+i_2-j_1,j_2,(j_1-i_2+[1))\cap[0))\delta_k, & \hbox{if~} j_1<i_2; \\
       (i_1,j_2,[1)\cap[0))\delta_k,                   & \hbox{if~} j_1=i_2; \\
       (i_1,j_1+j_2-i_2,[1)\cap(i_2-j_1+[0)))\delta_k, & \hbox{if~} j_1>i_2
     \end{array}
   \right.
   \\
   &=
   \left\{
     \begin{array}{ll}
       (i_1+i_2-j_1,j_2,[0))\delta_k, & \hbox{if~} j_1<i_2; \\
       (i_1,j_2,[1))\delta_k,         & \hbox{if~} j_1=i_2; \\
       (i_1,j_1+j_2-i_2,[1))\delta_k, & \hbox{if~} j_1>i_2
     \end{array}
   \right.
   \\
   &=
   \left\{
     \begin{array}{ll}
       (k(i_1+i_2-j_1),kj_2,[0)),       & \hbox{if~} j_1+1<i_2; \\
       (k(i_1+i_2-j_1),kj_2,[0)),       & \hbox{if~} j_1+1=i_2; \\
       (k(i_1+1),k(j_2+1),[0)),         & \hbox{if~} j_1=i_2; \\
       (k(i_1+1),k(j_1+j_2-i_2+1),[0)), & \hbox{if~} j_1>i_2
     \end{array}
   \right.
   \\
   &=
   \left\{
     \begin{array}{ll}
       (k(i_1+i_2-j_1),kj_2,[0)),       & \hbox{if~} j_1+1<i_2; \\
       (k(i_1+1),kj_2,[0)),             & \hbox{if~} j_1+1=i_2; \\
       (k(i_1+1),k(j_2+1),[0)),         & \hbox{if~} j_1=i_2; \\
       (k(i_1+1),k(j_1+j_2-i_2+1),[0)), & \hbox{if~} j_1>i_2.
     \end{array}
   \right.
\end{align*}
This completes the proof of the statement of the proposition.
\end{proof}

\begin{remark}\label{remark-2.7}
It obvious that if $\mathfrak{e}$ is the annihilating endomorphism of the monoid $\boldsymbol{B}_{\omega}^{\mathscr{F}}$ then $\mathfrak{e}=\gamma_0=\delta_0$.  
\end{remark}

By $\boldsymbol{End}_0^*(\boldsymbol{B}_{\omega}^{\mathscr{F}})$ we denote the semigroup of all non-injective monoid endomorphisms of the monoid $\boldsymbol{B}_{\omega}^{\mathscr{F}}$ for the family $\mathscr{F}=\{[0),[1)\}$.

\smallskip

Theorems~\ref{theorem-2.8} and~\ref{theorem-2.9} describe the algebraic structure of the semigroup $\boldsymbol{End}_0^*(\boldsymbol{B}_{\omega}^{\mathscr{F}})$.

\begin{theorem}\label{theorem-2.8}
If $\mathscr{F}=\{[0), [1)\}$, then for any non-injective monoid endomorphism $\mathfrak{e}$ of the monoid $\boldsymbol{B}_{\omega}^{\mathscr{F}}$ only one of the following conditions holds:
\begin{enumerate}
  \item\label{theorem-2.8(1)} $\mathfrak{e}$ is the annihilating endomorphism, i.e., $\mathfrak{e}=\gamma_0=\delta_0$;
  \item\label{theorem-2.8(2)} $\mathfrak{e}=\gamma_k$ for some positive integer $k$;
  \item\label{theorem-2.8(3)} $\mathfrak{e}=\delta_k$ for some positive integer $k$.
\end{enumerate}
\end{theorem}

\begin{proof}
Fix an arbitrary  non-injective monoid endomorphism $\mathfrak{e}$ of the monoid $\boldsymbol{B}_{\omega}^{\mathscr{F}}$. If $\mathfrak{e}$ is the annihilating endomorphism then statement \eqref{theorem-2.8(1)} holds. Hence, later we assume that the endomorphism $\mathfrak{e}$ is not annihilating.

\smallskip

By Lemma~\ref{lemma-2.2} the restriction $\mathfrak{e}{\upharpoonleft}_{\boldsymbol{B}_{\omega}^{\{[0)\}}}\boldsymbol{B}_{\omega}^{\{[0)\}}\to \boldsymbol{B}_{\omega}^{\mathscr{F}}$ of the endomorphism $\mathfrak{e}$ is an injective mapping. Since by Proposition~3 of \cite{Gutik-Mykhalenych=2020} the subsemigroup $\boldsymbol{B}_{\omega}^{\{[0)\}}$  of $\boldsymbol{B}_{\omega}^{\mathscr{F}}$ are isomorphic to the bicyclic semigroup, the injectivity of the restriction $\mathfrak{e}{\upharpoonleft}_{\boldsymbol{B}_{\omega}^{\{[0)\}}}$ of the endomorphism $\mathfrak{e}$, Proposition~4 of \cite{Gutik-Mykhalenych=2021}, and Lemma~2 of \cite{Gutik-Prokhorenkova-Sekh=2021} imply that there exists a positive integer $k$ such that
\begin{equation}\label{eq-2.1}
  (i,j,[0))\mathfrak{e}=(ki,kj,[0)),
\end{equation}
for all $i,j\in\omega$.

\smallskip

By Lemma~\ref{lemma-2.2} the restriction $\mathfrak{e}{\upharpoonleft}_{\boldsymbol{B}_{\omega}^{\{[1)\}}}\boldsymbol{B}_{\omega}^{\{[1)\}}\to \boldsymbol{B}_{\omega}^{\mathscr{F}}$ of the endomorphism $\mathfrak{e}$ is an injective mapping, and by Lemma~\ref{lemma-2.3} we have that $(\boldsymbol{B}_{\omega}^{\{[1)\}})\mathfrak{e}\subseteq \boldsymbol{B}_{\omega}^{\{[0)\}}$. By Proposition~1.4.21$(6)$ of \cite{Lawson=1998} a homomorphism of inverse semigroups preserves the natural partial order, and hence the following inequalities
\begin{equation*}
  (1,1,[0))\preccurlyeq(0,0,[1))\preccurlyeq(0,0,[0)),
\end{equation*}
Lemma~\ref{lemma-2.3}, and Propositions~2 of \cite{Gutik-Mykhalenych=2021}
imply that
\begin{align*}
  (k,k,[0))&=(1,1,[0))\mathfrak{e}\preccurlyeq \\
   &\preccurlyeq(s,s,[0))= \\
   &=(0,0,[1))\mathfrak{e}\preccurlyeq \\
   &\preccurlyeq (0,0,[0))=\\
   &=(0,0,[0))\mathfrak{e}
\end{align*}
for some $s\in\{0,1,\ldots,k\}$. Again by Proposition~1.4.21$(6)$ of \cite{Lawson=1998} and by Lemma~\ref{lemma-2.3} we get that
\begin{equation*}
  (1,1,[1))\mathfrak{e}=(s+p,s+p,[0))
\end{equation*}
for some non-negative integer $p$. If $p=0$ then $(1,1,[1))\mathfrak{e}=(0,0,[1))\mathfrak{e}$. By Lemma~\ref{lemma-2.2} the endomorphism $\mathfrak{e}$ is annihilating. Hence we assume that $p$ is a positive integer.

\smallskip

Let $(0,1,[1))\mathfrak{e}=(x,y,[0))$ for some $x,y\in\omega$. By Proposition~1.4.21$(1)$ of \cite{Lawson=1998} and Lemma~4 of \cite{Gutik-Mykhalenych=2020} we have that
\begin{align*}
  (1,0,[1))\mathfrak{e}&=((0,1,[1))^{-1})\mathfrak{e}= \\
   &= ((0,1,[1))\mathfrak{e})^{-1}= \\
   &=(x,y,[0))^{-1}= \\
   &=(y,x,[0)).
\end{align*}
Since
\begin{equation*}
(0,1,[1))\cdot(1,0,[1))=(0,0,[1)) \qquad \hbox{and} \qquad (1,0,[1))\cdot(0,1,[1))=(1,1,[1)),
\end{equation*}
the equalities $(0,0,[1))\mathfrak{e}=(s,s,[0))$ and $(1,1,[1))\mathfrak{e}=(s+p,s+p,[0))$ imply that
\begin{align*}
  (s,s,[0))&=(0,0,[1))\mathfrak{e}= \\
   &=((0,1,[1))\cdot(1,0,[1)))\mathfrak{e}= \\
   &=(0,1,[1))\mathfrak{e}\cdot(1,0,[1))\mathfrak{e}= \\
   &=(x,y,[0))\cdot(y,x,[0))= \\
   &=(x,x,[0))
\end{align*}
and
\begin{align*}
  (s+p,s+p,[0))&=(1,1,[1))\mathfrak{e}= \\
   &=((1,0,[1))\cdot(0,1,[1)))\mathfrak{e}= \\
   &=(1,0,[1))\mathfrak{e}\cdot(0,1,[1))\mathfrak{e}= \\
   &=(y,x,[0))\cdot(x,y,[0))= \\
   &=(y,y,[0)).
\end{align*}
This and the definition of the semigroup $\boldsymbol{B}_{\omega}^{\mathscr{F}}$ imply that
\begin{equation*}
  (0,1,[1))\mathfrak{e}=(s,s+p,[0)) \qquad \hbox{and} \qquad (1,0,[1))\mathfrak{e}=(s+p,s,[0)).
\end{equation*}
Then for any positive integers $n_1$ and $n_2$ by usual calculations we get that
\begin{align*}
  (0,n_1,[1))\mathfrak{e}&=(\underbrace{(0,1,[1))\cdot\ldots\cdot(0,1,[1))}_{n_1\hbox{\tiny{-times}}})\mathfrak{e}= \\
   &= \underbrace{(0,1,[1))\mathfrak{e}\cdot\ldots\cdot(0,1,[1))\mathfrak{e}}_{n_1\hbox{\tiny{-times}}}= \\
   &= (s,s+p,[0))^{n_1}= \\
   &=(s,s+n_1p,[0))
\end{align*}
and
\begin{align*}
  (n_2,0,[1))\mathfrak{e}&=(\underbrace{(1,0,[1))\cdot\ldots\cdot(1,0,[1))}_{n_2\hbox{\tiny{-times}}})\mathfrak{e}= \\
   &= \underbrace{(1,0,[1))\mathfrak{e}\cdot\ldots\cdot(1,0,[1))\mathfrak{e}}_{n_2\hbox{\tiny{-times}}}=
   \\
   &= (s+p,s,[0))^{n_2}= \\
   &=(s+n_2p,s,[0)),
\end{align*}
and hence
\begin{equation}\label{eq-2.2}
  (n_1,n_2,[1))\mathfrak{e}=(s+n_1p,s+n_2p,[0)).
\end{equation}

The definition of the natural partial order on the semigroup $\boldsymbol{B}_{\omega}^{\mathscr{F}}$ (see Proposition~4 of \cite{Gutik-Mykhalenych=2021}) imply that for any positive integer $m$ we have that
\begin{equation*}
  (m+1,m+1,[0))\preccurlyeq(m,m,[1))\preccurlyeq(m,m,[0)).
\end{equation*}
Then by equalities \eqref{eq-2.1}, \eqref{eq-2.2}, and Proposition~1.4.21$(6)$ of \cite{Lawson=1998} we obtain that
\begin{align*}
  (k(m+1),k(m+1),[0))&=(m+1,m+1,[0))\mathfrak{e}\preccurlyeq \\
   &\preccurlyeq(s+pm,s+pm,[0))=\\
   &=(m,m,[1))\mathfrak{e}\preccurlyeq \\
   &\preccurlyeq(m,m,[0))\mathfrak{e}=\\
   &=(km,km,[0)).
\end{align*}
The above inequalities  and the definition of the natural partial order on the semigroup $\boldsymbol{B}_{\omega}^{\mathscr{F}}$ (see Proposition~4 of \cite{Gutik-Mykhalenych=2021}) imply that $km\leqslant s+pm\leqslant k(m+1)$ for any positive integer $m$. This implies that
\begin{equation*}
  k\leqslant \frac{s}{m}+p\leqslant k+\frac{1}{m},
\end{equation*}
and since $p$ is a positive integer we get that $p=k$. Hence by \eqref{eq-2.2} we get that
\begin{equation}\label{eq-2.3}
  (n_1,n_2,[1))\mathfrak{e}=(s+n_1k,s+n_2k,[0)),
\end{equation}
for all $n_1,n_2\in\omega$.

\smallskip

It is obvious that if $s\in\{1,\ldots,k-1\}$ then $\mathfrak{e}$ is an injective monoid endomorphism of the semigroup. Hence we have that either $s=0$ or $s=k$, Simple verifications show that
\begin{equation*}
\mathfrak{e}=
\left\{
  \begin{array}{ll}
    \gamma_k, & \hbox{if~} s=0;\\
    \delta_k, & \hbox{if~} s=k.
  \end{array}
\right.
\end{equation*}
This completes the proof of the theorem.
\end{proof}

\begin{theorem}\label{theorem-2.9}
Let $\mathscr{F}=\{[0), [1)\}$. Then for all positive integers $k_1$ and $k_2$ the following conditions hold:
\begin{enumerate}
  \item\label{theorem-2.9(1)} $\gamma_{k_1}\gamma_{k_2}=\gamma_{k_1k_2}$;
  \item\label{theorem-2.9(2)} $\gamma_{k_1}\delta_{k_2}=\gamma_{k_1k_2}$;
  \item\label{theorem-2.9(3)} $\delta_{k_1}\gamma_{k_2}=\delta_{k_1k_2}$;
  \item\label{theorem-2.9(4)} $\delta_{k_1}\delta_{k_2}=\delta_{k_1k_2}$.
\end{enumerate}
\end{theorem}

\begin{proof}[\textsl{Proof}]
\eqref{theorem-2.9(1)}
For any $i,j\in\omega$ we have that
\begin{align*}
  (i,j,[0))\gamma_{k_1}\gamma_{k_2}&=(k_1i,k_1j,[0))\gamma_{k_2}=\\
            &=(k_1k_2i,k_1k_2j,[0)),
\end{align*}
and $(i,j,[1))\gamma_{k_1}=(i,j,[0))\gamma_{k_1}$. This implies  that $\gamma_{k_1}\gamma_{k_2}=\gamma_{k_1k_2}$.

\medskip
\eqref{theorem-2.9(2)}
Since
\begin{align*}
  (i,j,[0))\gamma_{k_1}\delta_{k_2}&=(k_1i,k_1j,[0))\delta_{k_2}=\\
    &=(k_1k_2i,k_1k_2j,[0)),
\end{align*}
and $(i,j,[1))\gamma_{k_1}=(i,j,[0))\gamma_{k_1}$ for all $i,j\in\omega$, we get that $\gamma_{k_1}\delta_{k_2}=\gamma_{k_1k_2}$.

\medskip
\eqref{theorem-2.9(3)}
For any $i,j\in\omega$ we have that
\begin{align*}
  (i,j,[0))\delta_{k_1}\gamma_{k_2}&=(k_1i,k_1j,[0))\gamma_{k_2}=\\
    &=(k_1k_2i,k_1k_2j,[0)),
\end{align*}
and
\begin{align*}
  (i,j,[1))\delta_{k_1}\gamma_{k_2}&=(k_1(i+1),k_1(j+1),[0))\gamma_{k_2}=\\
    &=(k_1k_2(i+1),k_1k_2(j+1),[0)),
\end{align*}
and hence  $\delta_{k_1}\gamma_{k_2}=\delta_{k_1k_2}$.

\medskip
\eqref{theorem-2.9(4)}
For any $i,j\in\omega$ we have that
\begin{align*}
  (i,j,[0))\delta_{k_1}\delta_{k_2}&=(k_1i,k_1j,[0))\delta_{k_2}=\\
     &=(k_1k_2i,k_1k_2j,[0)),
\end{align*}
and
\begin{align*}
  (i,j,[1))\delta_{k_1}\delta_{k_2}&=(k_1(i+1),k_1(j+1),[0))\delta_{k_2}=\\
     &=(k_1k_2(i+1),k_1k_2(j+1),[0)),
\end{align*}
and hence  $\delta_{k_1}\delta_{k_2}=\delta_{k_1k_2}$.
\end{proof}

By $\mathfrak{e}_{\boldsymbol{0}}$ we denote the annihilating monoid endomorphism of the monoid $\boldsymbol{B}_{\omega}^{\mathscr{F}}$ for the family \linebreak $\mathscr{F}=\{[0),[1)\}$, i.e., $(i,j,[p))\mathfrak{e}_{\boldsymbol{0}}=(0,0,[0))$ for all $i,j\in\omega$ and $p=0,1$. We put $\boldsymbol{End}^*(\boldsymbol{B}_{\omega}^{\mathscr{F}})=\boldsymbol{End}_0^*(\boldsymbol{B}_{\omega}^{\mathscr{F}})\setminus \{\mathfrak{e}_{\boldsymbol{0}}\}$. Theorem~\ref{theorem-2.9} implies that $\boldsymbol{End}^*(\boldsymbol{B}_{\omega}^{\mathscr{F}})$ is a subsemigroup of $\boldsymbol{End}_0^*(\boldsymbol{B}_{\omega}^{\mathscr{F}})$.

\smallskip

Theorem~\ref{theorem-2.9} implies the following corollary.

\begin{corollary}
If $\mathscr{F}=\{[0), [1)\}$, then the elements $\gamma_1$ and $\delta_1$ are unique idempotents of the semigroup $\boldsymbol{End}^*(\boldsymbol{B}_{\omega}^{\mathscr{F}})$.
\end{corollary}

Next, by ${\mathfrak{LZ}}_2$ we denote the left zero semigroup with two elements and by $\mathbb{N}_u$ the multiplicative semigroup of positive integers.

\begin{proposition}\label{proposition-2.10}
Let $\mathscr{F}=\{[0), [1)\}$. Then the semigroup $\boldsymbol{End}^*(\boldsymbol{B}_{\omega}^{\mathscr{F}})$ is isomorphic to the direct product ${\mathfrak{LZ}}_2\times \mathbb{N}_u$.
\end{proposition}

\begin{proof} 
Put ${LZ}_2=\{c,d\}$. We define a map $\mathfrak{I}\colon \boldsymbol{End}^*(\boldsymbol{B}_{\omega}^{\mathscr{F}})\to {\mathfrak{LZ}}_2\times \mathbb{N}_u$ by the formula
\begin{equation*}
  (\mathfrak{e})\mathfrak{I}=
\left\{
  \begin{array}{ll}
    (c,k), & \hbox{if~} \mathfrak{e}=\gamma_k; \\
    (d,k), & \hbox{if~} \mathfrak{e}=\delta_k.
  \end{array}
\right.
\end{equation*}
It is obvious that such defined map $\mathfrak{I}$ is bijective, and by Theorem~\ref{theorem-2.9} it is a homomor\-ph\-ism.
\end{proof}


Theorem~\ref{theorem-2.11} describes Green's relations on the semigroup $\boldsymbol{End}^*(\boldsymbol{B}_{\omega}^{\mathscr{F}})$. Later by $\boldsymbol{End}^*(\boldsymbol{B}_{\omega}^{\mathscr{F}})^1$ we denote the semigroup $\boldsymbol{End}^*(\boldsymbol{B}_{\omega}^{\mathscr{F}})$ with adjoined identity element.

\begin{theorem}\label{theorem-2.11}
Let $\mathscr{F}=\{[0), [1)\}$. Then the following statements hold:
\begin{enumerate}
  \item\label{theorem-2.11(1)} $\gamma_{k_1}\mathscr{R}\gamma_{k_2}$ in $\boldsymbol{End}^*(\boldsymbol{B}_{\omega}^{\mathscr{F}})$ if and only if $k_1=k_2$;

  \item\label{theorem-2.11(2)} $\gamma_{k_1}\mathscr{R}\delta_{k_2}$ does not hold in $\boldsymbol{End}^*(\boldsymbol{B}_{\omega}^{\mathscr{F}})$ for any $\gamma_{k_1},\delta_{k_2}$;

  \item\label{theorem-2.11(3)} $\delta_{k_1}\mathscr{R}\delta_{k_2}$ in $\boldsymbol{End}^*(\boldsymbol{B}_{\omega}^{\mathscr{F}})$ if and only if $k_1=k_2$;

  \item\label{theorem-2.11(4)} $\gamma_{k_1}\mathscr{L}\gamma_{k_2}$ in $\boldsymbol{End}^*(\boldsymbol{B}_{\omega}^{\mathscr{F}})$ if and only if $k_1=k_2$;

  \item\label{theorem-2.11(5)} $\gamma_{k_1}\mathscr{L}\delta_{k_2}$ in $\boldsymbol{End}^*(\boldsymbol{B}_{\omega}^{\mathscr{F}})$ if and only if $k_1=k_2$;

  \item\label{theorem-2.11(6)} $\delta_{k_1}\mathscr{L}\delta_{k_2}$ in $\boldsymbol{End}^*(\boldsymbol{B}_{\omega}^{\mathscr{F}})$ if and only if $k_1=k_2$;

  \item\label{theorem-2.11(7)} $\mathscr{H}$ is the identity relation on $\boldsymbol{End}^*(\boldsymbol{B}_{\omega}^{\mathscr{F}})$;

  \item\label{theorem-2.11(8)} $\mathfrak{e}_1\mathscr{D}\mathfrak{e}_2$ in $\boldsymbol{End}^*(\boldsymbol{B}_{\omega}^{\mathscr{F}})$ if and only if $\mathfrak{e}_1=\mathfrak{e}_2$ or there exists a positive integer $k$ such that $\mathfrak{e}_1,\mathfrak{e}_2\in \{\gamma_k,\delta_k\}$;

  \item\label{theorem-2.11(9)} $\mathscr{D}=\mathscr{J}$ in $\boldsymbol{End}^*(\boldsymbol{B}_{\omega}^{\mathscr{F}})$.
\end{enumerate}
\end{theorem}

\begin{proof}[\textsl{Proof}]
\eqref{theorem-2.11(1)} $(\Rightarrow)$ Suppose that $\gamma_{k_1}\mathscr{R}\gamma_{k_2}$ in $\boldsymbol{End}^*(\boldsymbol{B}_{\omega}^{\mathscr{F}})$. Then there exist $\mathfrak{e}_1,\mathfrak{e}_2\in\boldsymbol{End}^*(\boldsymbol{B}_{\omega}^{\mathscr{F}})^1$ such that $\gamma_{k_1}=\gamma_{k_2}\mathfrak{e}_1$ and $\gamma_{k_2}=\gamma_{k_1}\mathfrak{e}_2$. The equality $\gamma_{k_1}=\gamma_{k_2}\mathfrak{e}_1$ and Theorem~\ref{theorem-2.9} imply that there exists a positive integer $p$ such that either $\mathfrak{e}_1=\gamma_p$ or $\mathfrak{e}_1=\delta_p$. In both above cases by Theorem~\ref{theorem-2.9} we have that
\begin{equation*}
  \gamma_{k_1}=\gamma_{k_2}\mathfrak{e}_1=\gamma_{k_2}\gamma_p=\gamma_{k_2}\delta_p=\gamma_{k_2p},
\end{equation*}
and hence $k_2|k_1$. The proof of the statement that $\gamma_{k_2}=\gamma_{k_1}\mathfrak{e}_2$ implies that $k_1|k_2$ is similar. Therefore we get that $k_1=k_2$.

\smallskip

Implication $(\Leftarrow)$ is trivial.

\medskip
Statement \eqref{theorem-2.11(2)} follows from Theorem~\ref{theorem-2.9}\eqref{theorem-2.9(2)}.

\medskip
The proof of statement \eqref{theorem-2.11(3)} is similar to \eqref{theorem-2.11(1)}.

\medskip
\eqref{theorem-2.11(4)} $(\Rightarrow)$ Suppose that $\gamma_{k_1}\mathscr{L}\gamma_{k_2}$ in $\boldsymbol{End}^*(\boldsymbol{B}_{\omega}^{\mathscr{F}})$. Then there exist $\mathfrak{e}_1,\mathfrak{e}_2\in\boldsymbol{End}^*(\boldsymbol{B}_{\omega}^{\mathscr{F}})^1$ such that $\gamma_{k_1}=\mathfrak{e}_1\gamma_{k_2}$ and $\gamma_{k_2}=\mathfrak{e}_2\gamma_{k_1}$. The equality $\gamma_{k_1}=\mathfrak{e}_1\gamma_{k_2}$ and Theorem~\ref{theorem-2.9} imply that there exists a positive integer $p$ such that $\mathfrak{e}_1=\gamma_p$. Then we have that
\begin{equation*}
  \gamma_{k_1}=\mathfrak{e}_1\gamma_{k_2}=\gamma_p\gamma_{k_2}=\gamma_{p k_2},
\end{equation*}
and hence $k_2|k_1$. The proof of the statement that $\gamma_{k_2}=\mathfrak{e}_2\gamma_{k_1}$ implies that $k_1|k_2$ is similar. Therefore we get that $k_1=k_2$.

\smallskip

Implication $(\Leftarrow)$ is trivial.

\medskip
\eqref{theorem-2.11(5)} $(\Rightarrow)$ Suppose that $\gamma_{k_1}\mathscr{L}\delta_{k_2}$ in $\boldsymbol{End}^*(\boldsymbol{B}_{\omega}^{\mathscr{F}})$. Then there exist $\mathfrak{e}_1,\mathfrak{e}_2\in\boldsymbol{End}^*(\boldsymbol{B}_{\omega}^{\mathscr{F}})^1$ such that $\gamma_{k_1}=\mathfrak{e}_1\delta_{k_2}$ and $\delta_{k_2}=\mathfrak{e}_2\gamma_{k_1}$. The equality $\gamma_{k_1}=\mathfrak{e}_1\delta_{k_2}$ and Theorem~\ref{theorem-2.9} imply that there exists a positive integer $p$ such that $\mathfrak{e}_1=\gamma_p$. Then we have that
\begin{equation*}
  \gamma_{k_1}=\mathfrak{e}_1\delta_{k_2}=\gamma_p\delta_{k_2}=\gamma_{p k_2},
\end{equation*}
and hence $k_2|k_1$. The equality $\delta_{k_2}=\mathfrak{e}_2\gamma_{k_1}$ and Theorem~\ref{theorem-2.9} imply that there exists a positive integer $q$ such that $\mathfrak{e}_1=\delta_q$. Then we have that
\begin{equation*}
  \delta_{k_2}=\mathfrak{e}_2\gamma_{k_1}=\delta_q\gamma_{k_1}=\gamma_{qk_1},
\end{equation*}
and hence $k_1|k_2$. Thus we get that $k_1=k_2$.

\smallskip

Implication $(\Leftarrow)$ is trivial.

\medskip
\eqref{theorem-2.11(6)} $(\Rightarrow)$ Suppose that $\delta_{k_1}\mathscr{L}\delta_{k_2}$ in $\boldsymbol{End}^*(\boldsymbol{B}_{\omega}^{\mathscr{F}})$. Then there exist $\mathfrak{e}_1,\mathfrak{e}_2\in\boldsymbol{End}^*(\boldsymbol{B}_{\omega}^{\mathscr{F}})^1$ such that $\delta_{k_1}=\mathfrak{e}_1\delta_{k_2}$ and $\delta_{k_2}=\mathfrak{e}_2\delta_{k_1}$. The equality $\delta_{k_1}=\mathfrak{e}_1\delta_{k_2}$ and Theorem~\ref{theorem-2.9} imply that there exists a positive integer $p$ such that $\mathfrak{e}_1=\delta_p$. Then we have that
\begin{equation*}
  \delta_{k_1}=\mathfrak{e}_1\delta_{k_2}=\delta_p\delta_{k_2}=\delta_{p k_2},
\end{equation*}
and hence $k_2|k_1$. The proof of the statement that $\delta_{k_2}=\mathfrak{e}_2\delta_{k_1}$ implies that $k_1|k_2$ is similar. Hence we get that $k_1=k_2$.

\smallskip

Implication $(\Leftarrow)$ is trivial.

\medskip
\eqref{theorem-2.11(7)} By statements \eqref{theorem-2.11(1)}, \eqref{theorem-2.11(2)}, and \eqref{theorem-2.11(3)}, $\mathscr{R}$ is the identity relation on the semigroup $\boldsymbol{End}^*(\boldsymbol{B}_{\omega}^{\mathscr{F}})$. Then so is $\mathscr{H}$, because $\mathscr{H}\subseteq \mathscr{R}$.

\medskip
Statement \eqref{theorem-2.11(8)} follows from statements \eqref{theorem-2.11(1)}--\eqref{theorem-2.11(6)}.

\medskip
\eqref{theorem-2.11(9)} Suppose to the contrary that $\mathscr{D}\neq\mathscr{J}$ in $\boldsymbol{End}^*(\boldsymbol{B}_{\omega}^{\mathscr{F}})$. Since $\mathscr{D}\subseteq\mathscr{J}$, statement \eqref{theorem-2.11(8)} implies that there exist $\mathfrak{e}_1,\mathfrak{e}_2\in\boldsymbol{End}^*(\boldsymbol{B}_{\omega}^{\mathscr{F}})^1$ such that $\mathfrak{e}_1\mathscr{J}\mathfrak{e}_2$ and $\mathfrak{e}_1,\mathfrak{e}_2\notin \{\gamma_k,\delta_k\}$ for any positive integer $k$. Then there exist distinct positive integers $k_1$ and $k_2$ such that $\mathfrak{e}_1\in \{\gamma_{k_1},\delta_{k_1}\}$ and $\mathfrak{e}_2\in \{\gamma_{k_2},\delta_{k_2}\}$. Without loss of generality we may assume that $k_1<k_2$. Since $\mathfrak{e}_1\mathscr{J}\mathfrak{e}_2$ there exist $\mathfrak{e}_1^{\prime},\mathfrak{e}_2^{\prime}, \mathfrak{e}_1^{\prime\prime},\mathfrak{e}_2^{\prime\prime}\in\boldsymbol{End}^*(\boldsymbol{B}_{\omega}^{\mathscr{F}})^1$ such that $\mathfrak{e}_1=\mathfrak{e}_1^{\prime}\mathfrak{e}_2\mathfrak{e}_1^{\prime\prime}$ and $\mathfrak{e}_2=\mathfrak{e}_2^{\prime}\mathfrak{e}_2\mathfrak{e}_2^{\prime\prime}$. Since $\mathfrak{e}_1\in \{\gamma_{k_1},\delta_{k_1}\}$ and $\mathfrak{e}_2\in \{\gamma_{k_2},\delta_{k_2}\}$, the equality $\mathfrak{e}_1=\mathfrak{e}_1^{\prime}\mathfrak{e}_2\mathfrak{e}_1^{\prime\prime}$, Theorems~\ref{theorem-2.8} and ~\ref{theorem-2.9} imply that $k_2|k_1$. This contradicts the inequality $k_1<k_2$. The obtained contradiction implies the requested statement
\end{proof}

\begin{remark}\label{remark-2.13}
Since $\mathfrak{e}_{\boldsymbol{0}}$ is zero of the semigroup $\boldsymbol{End}_0^*(\boldsymbol{B}_{\omega}^{\mathscr{F}})$, the classes of equivalence of Green's relations of non-zero elements of $\boldsymbol{End}_0^*(\boldsymbol{B}_{\omega}^{\mathscr{F}})$ in the semigroup  $\boldsymbol{End}_0^*(\boldsymbol{B}_{\omega}^{\mathscr{F}})$ coincide with their corresponding  classes of equivalence in $\boldsymbol{End}^*(\boldsymbol{B}_{\omega}^{\mathscr{F}})$, and moreover we have that
\begin{equation*}
\boldsymbol{L}_{\mathfrak{e}_{\boldsymbol{0}}}=\boldsymbol{R}_{\mathfrak{e}_{\boldsymbol{0}}}=\boldsymbol{H}_{\mathfrak{e}_{\boldsymbol{0}}}=
\boldsymbol{D}_{\mathfrak{e}_{\boldsymbol{0}}}= \boldsymbol{J}_{\mathfrak{e}_{\boldsymbol{0}}}=\{\mathfrak{e}_{\boldsymbol{0}}\}
\end{equation*}
in the semigroup $\boldsymbol{End}_0^*(\boldsymbol{B}_{\omega}^{\mathscr{F}})$.
\end{remark}

\section*{\textbf{Acknowledgements}}

The authors acknowledge the referee for his/her comments and suggestions.

\end{document}